\documentclass[article, 11pt]{amsart} 
\usepackage{bbm}
\usepackage[top=2.65cm,left=2.7cm,right=2.7cm,bottom=2.65cm]{geometry}
\usepackage{amsfonts}
\usepackage{amsmath,amsthm,amscd,mathrsfs,amssymb,graphicx,enumerate,
    dsfont}
\usepackage{xypic}
\usepackage{hyperref}
\usepackage[usenames,dvipsnames]{xcolor}
\hypersetup{colorlinks=true,citecolor=NavyBlue,linkcolor=Brown,urlcolor=Orange}
\newtheorem{thm}{Theorem}[section]
\newtheorem{cor}[thm]{Corollary}

\newtheorem{prop}[thm]{Proposition}
\newtheorem{conj}[thm]{Conjecture}
\theoremstyle{definition}
\newtheorem{defn}[thm]{Definition}

\newtheorem{ex}[thm]{Example}
\newtheorem{rmk}[thm]{Remark}

\newtheorem{ques}[thm]{Question}
 
\DeclareMathOperator{\Gal}{Gal}

\DeclareMathOperator{\trdeg}{trdeg \,}

\newcommand{\C}{\ensuremath\mathbb{C}}
\newcommand{\R}{\ensuremath\mathbb{R}}
\newcommand{\Z}{\ensuremath\mathbb{Z}}
\newcommand{\Q}{\ensuremath\mathbb{Q}}

\newcommand{\HH}{\ensuremath\mathrm{H}}

\newcommand{\B}{\ensuremath\mathrm{B}}
\newcommand{\dR}{\ensuremath\mathrm{dR}}

\newcommand{\mot}{\ensuremath\mathrm{mot}}
\newcommand{\Mot}{\ensuremath\mathrm{Mot}}
\newcommand{\Hdg}{\ensuremath\mathrm{Hdg}}

\renewcommand{\And}{\ensuremath\mathrm{And}}
\newcommand{\an}{\ensuremath\mathrm{an}}

\newif\ifHideFoot
\HideFoottrue  % This line can be left alone.
\HideFootfalse  % Comment this line to hide footnotes
\ifHideFoot

\newcommand{\Tobias}[1]{}
\else 
\newcommand{\marg}[1]{\normalsize{{
            \color{red}\footnote{{\color{blue}#1}}}{\marginpar[\vskip
            -.25cm{\color{red}\hfill$\Rightarrow$\tiny\thefootnote}]{\vskip
                -.2cm{\color{red}$\Leftarrow$\tiny\thefootnote}}}}}

\newcommand{\Tobias}[1]{\marg{(Tobias) #1}}

\fi
\AtBeginDocument{%
    \def\MR#1{}
}
\begin{document}
 \title{Hodge structures not coming from geometry}
 \begin{abstract}
 Hodge theory associates to a smooth projective variety over $\C$ a piece of linear algebra information, called a $\Q$-Hodge structure. 
 Conversely, it is a natural question which abstract $\Q$-Hodge structures arise from the cohomology of a smooth projective complex variety, or more generally, from a pure motive over $\C$.
 By a classical argument involving Griffiths transversality and a Baire category argument, it is well known that there are many Hodge structures which do not come from geometry in this sense.
 However, the argument is not constructive, and does not seem to give a criterion to decide whether a given Hodge structure comes from geometry.
  We formulate an intrinsic condition on a $\Q$-Hodge structure that we expect to be satisfied for all Hodge structures coming from geometry.
  We prove that this expectation follows from the conjunction of two fundamental conjectures in Hodge theory and transcendence theory:
the conjecture that Hodge cycles are motivated and André's generalized Grothendieck period conjecture.
  By doing so, we exhibit explicit examples of $\Q$-Hodge structures which should not come from geometry.
\end{abstract}
   
    \author{Tobias Kreutz}

    \address{Max Planck Institute for Mathematics, Bonn (Germany)}
    \email{kreutz@mpim-bonn.mpg.de}

    \date{\today}

\maketitle

\section*{Introduction}

Let $X$ be a smooth projective variety over $\C$. 
The fundamental insight of Hodge theory is that Betti cohomology $\HH_\B^n(X, \Q) = \HH^n_{\mathrm{sing}}(X^{\an}, \Q)$ carries a Hodge structure:
there is a canonical \emph{Hodge decomposition} of complex vector spaces
$$ \HH_\B^n(X, \Q) \otimes_\Q \C = \bigoplus_{p+q=n} \HH^{p,q}(X)$$
with the property that $\overline{\HH^{p,q}(X)} = \HH^{q,p}(X)$.

Hodge theory therefore associates to the geometric object $X$ a linear algebra datum which can be abstracted into the following definition:

\begin{defn}
A \emph{pure $\Q$-Hodge structure of weight $n$} is a finite dimensional $\Q$-vector space $H$ together with a decomposition of complex vector spaces
$$H \otimes_\Q \C =\bigoplus_{p+q=n} H^{p,q}$$
such that $\overline{H^{p,q}} = H^{q,p}$.

\end{defn}

\subsection{Hodge structures coming from geometry}
In this note, we are interested in necessary conditions for an abstract Hodge structure $H$ to arise as the cohomology $\HH_\B^n(X, \Q)$ of a smooth projective variety $X$ over $\C$.
In fact, we will ask the slightly more general question which Hodge structures arise as the realization of a pure \emph{motive} over $\C$. We use the framework of André motives as developed in \cite{AndreIHES}, cf. section \ref{sectionAndre}.

\begin{defn}
We say that a $\Q$-Hodge structure $H$ \emph{comes from geometry} if it is isomorphic to the Hodge realization of an André motive over $\C$.
\end{defn}
If $H$ is irreducible, then modulo the conjecture that Hodge cycles are motivated, this is the same as saying that $H$ is a Tate twist of a direct summand of the Hodge structure $\HH_\B^n(X, \Q)$ of some smooth projective variety $X$ over $\C$ for some $n$.

The aim of this note is to discuss the following question:
\begin{ques}
Which $\Q$-Hodge structures $H$ come from geometry?
\end{ques}

The first observation is that $H$ is necessarily polarizable.
Recall that a polarization of a pure Hodge structure $H$ of weight $n$ is a morphism of Hodge structures $H \otimes_\Q H \to \Q(-n)$ which satisfies a positivity property on certain pieces of the Hodge structure.
We therefore assume all of our Hodge structures to have this property.

One may ask whether all polarizable Hodge structures come from geometry.
This turns out to be far from true:
using Griffiths transversality and a Baire category argument, it is known that most polarizable Hodge structures do not come from geometry (cf. section \ref{distribution}).
However, the argument is not constructive: it does not give examples of Hodge structures which are not expected to come from geometry.

In \cite[page 300, footnote \ddag]{GrothendieckHodge}, Grothendieck writes:

\begin{quote}
It seems that there is no necessary intrinsic condition known for an abstract Hodge structure to be embeddable in one coming from a projective smooth scheme over $\C$, except the existence of a "polarization" -- although (as Mumford pointed out to me) Griffiths's general transversality theorem implies (by a Baire argument) that there are many Hodge structures of given degree $\ge 2$ which are not "algebraical" in the previous sense. Of course, any necessary condition of algebraicity would be highly interesting!
\end{quote}

\subsection{A Hodge theoretic condition}

The goal of this paper is to formulate a conjectural necessary intrinsic condition for a Hodge structure to come from geometry. 
By doing so, we give explicit examples of Hodge structures which should not come from geometry.

To formulate the condition, we define two invariants of the $\Q$-Hodge structure $H$: 
the \emph{transcendence degree} $\trdeg H$ and the \emph{horizontal codimension} $\mathrm{hcodim \,} H$ (cf. Definition \ref{trdeg} and Definition \ref{defnhorizontalcodimension}).
We conjecture the following: 

\begin{conj}\label{conj}
Let $H$ be a $\Q$-Hodge structure such that
\begin{equation} \label{inequality} \trdeg H < \mathrm{hcodim \,} H. \end{equation}
Then $H$ does not come from geometry.
\end{conj}

\begin{rmk}
There exist abstract $\Q$-Hodge structures which satisfy the inequality (\ref{inequality}), see section \ref{sectionexamples}.
\end{rmk}

The transcendence degree of $H$ is defined to be the minimal transcendence degree of a field of definition of the Hodge filtration.
On the other hand, the horizontal codimension $\mathrm{hcodim \,} H$ measures the constraint of Griffiths transversality on the classifying space of $H$.
The precise definitions will be given in section \ref{sectioninvariants}.
We give two examples to illustrate the nature of these invariants and of Conjecture \ref{conj}:

\begin{ex}
\begin{enumerate}[(i)]
\item
If $H$ is a weight $1$ Hodge structure of type $(1,0), (0,1)$, then $\mathrm{hcodim \,} H=0$. Hence Conjecture \ref{conj} poses no restrictions in this case. This is consistent with the fact that all polarizable Hodge structures of type $(1,0), (0,1)$ come from abelian varieties.
Conversely, we have $\mathrm{hcodim \,} H=0$ if and only if the classifying space for the Hodge structure $H$ is a Shimura variety. Note that in this case Griffiths transversality is a trivial constraint.
\item 
On the other hand, $\trdeg H =0$ is equivalent to the Hodge filtration being defined over the algebraic closure $\overline \Q$. 
Conjecture \ref{conj} predicts that $H$ does not come from geometry unless $\mathrm{hcodim \,} H=0$, i.e. the classifying space for the Hodge structure $H$ is a Shimura variety.
\end{enumerate}
\end{ex}

We show that Conjecture \ref{conj} follows from the conjunction of two fundamental conjectures in Hodge theory and transcendence theory:
the conjecture that Hodge cycles are motivated and André's generalized Grothendieck period conjecture.

\begin{thm}[Theorem \ref{thmHSnotgeometric}]\label{thmimplication}
Assume that Hodge cycles are motivated (Conjecture \ref{conjmotivated}) and André's generalized Grothendieck period conjecture \ref{conjgeneralizedperiod} holds.
Then Conjecture \ref{conj} holds true. \end{thm}

Unfortunately, we are not able to give a single example of a polarizable Hodge structure for which one can unconditionally prove that it does not come from geometry.
\subsection{Idea of proof}
Let $H$ be a Hodge structure which is the Hodge realization of some André motive $M$ over $\C$.
We have to show that
$\trdeg H \ge \mathrm{hcodim \,} H$.
The proof of Theorem \ref{thmimplication} proceeds in two steps:

\begin{enumerate}[(i)]
\item Let $K\subset \C$ be a finitely generated field of definition for $M$.
Firstly, using André's generalized period conjecture one shows a lower bound for $\trdeg H$ depending on the transcendence degree of $K$ and data associated with the Hodge structure $H$ (cf. Corollary \ref{corgeneralized}). The fact that André's conjecture implies such a bound can be found in \cite[23.4.4]{AndreBook} (in the case of abelian varieties).
\item 
Secondly, we produce an upper bound for the transcendence degree of the field $K$ in terms of invariants attached to $H$ (cf. Proposition \ref{propdescent}).
This is done by spreading out $M$ to a family of motives $\mathbb{M} \to S$ defined over $\overline \Q$.
We look at the period map $$\varphi: S_\C^{\an} \to \mathcal{D} / \Gamma $$ corresponding to the associated variation of Hodge structure.
Assuming that Hodge cycles are motivated, Griffiths transversality implies that the dimension of its image is controlled by $\mathrm{hcodim \,} H$. 
This image is an algebraic variety by \cite{BBT}, and can be defined over $\overline \Q$ up to a quasi-finite map.
Using these properties, one finds a point  $x \in S(K')$ such that the motive $\mathbb{M}_x$ has Hodge realization $H$, and 
is defined over a field $K'$ whose transcendence degree is bounded by invariants attached to the Hodge structure $H$. 

\end{enumerate}
Together, the bounds in $(i)$ and $(ii)$ create a tension between the invariants of the Hodge structure $H$, which forces the inequality $\trdeg H \ge \mathrm{hcodim \,} H$ (cf. Theorem \ref{thmHSnotgeometric}).

\subsection{Analogy with the Fontaine-Mazur conjecture}
While the condition \begin{equation}\label{ge}
  \trdeg H \ge \mathrm{hcodim \,} H  
\end{equation} should be a necessary condition for a Hodge structure to come from geometry, we cannot expect it to be sufficient.
Namely, it allows way more Hodge structures than we can expect to come from geometry, see the discussion in section \ref{distribution}.
In addition, the category of Hodge structures coming from geometry forms a Tannakian subcategory of the Tannakian category $\Q-\mathrm{HS}$ of Hodge structures.
However, it seems that condition (\ref{ge}) is not stable under direct sums, direct summands, or tensor products.

We want to mention the following analogy with Galois representations.
Let $V$ be a finite dimensional $\Q_p$-vector space and $\rho: \Gal(\overline K/K) \to \mathrm{GL}(V)$ a continuous representation of the absolute Galois group of a number field $K$.
Then there are strong necessary conditions for $V$ to come from the étale cohomology $\HH^{n}_{\mathrm{\acute{e}t}}(X_{\overline K}, \Q_p)$ of a smooth projective variety over $K$:
the representation $V$ has to be \emph{de Rham} at all places $\mathfrak{p}$ of $K$ lying over $p$, and unramified at almost all other places.
The \emph{Fontaine-Mazur conjecture} \cite[Conjecture 1]{FontaineMazur} asserts that these conditions should in a certain sense also be sufficient: an irreducible Galois representation $\rho$ which satisfies these properties should be a Tate twist of a subquotient of some $\HH^{n}_{\mathrm{\acute{e}t}}(X_{\overline K}, \Q_p)$.

Of course, it would be very interesting to formulate sufficient conditions for an abstract Hodge structure $H$ to come from geometry, giving a conjectural description of the category of Hodge structures coming from geometry and a Hodge theoretic analog of the Fontaine-Mazur conjecture.

\subsection*{Acknowledgements}
We thank Mingmin Shen and Charles Vial for interesting discussions in Amsterdam in March 2023 and Bruno Klingler for his comments and remarks on a previous version of this paper.    

\section{Hodge structures and André motives}
\subsection{Invariants of Hodge structures}\label{sectioninvariants}

Given a finite dimensional $\Q$-vector space $H$, defining a pure Hodge structure of weight $n$ on $H$ is equivalent to giving a real algebraic representation $\rho: \mathbb{S} \to \mathrm{GL}(H_\R)$ of the Deligne torus $\mathbb{S} := \mathrm{Res}_{\C/\R} \mathbb{G}_m$ whose restriction to $\mathbb{G}_{m, \R} \subset \mathbb{S}$ is given by $t \mapsto t^{-n}$.

\begin{defn}
The \emph{Mumford-Tate group} of the Hodge structure $H$ is the smallest algebraic $\Q$-subgroup $G$ of $\mathrm{GL}(H)$ such that $G_\R$ contains the image of $\rho$. 
\end{defn}
\begin{rmk}
\begin{enumerate}[(i)]
\item
One can see from the definition that the Mumford-Tate group is always connected.
If the Hodge structure $H$ is polarizable, the Mumford-Tate group $G$ is a reductive group. 
\item 
The category $\Q-\mathrm{HS}$ of finite direct sums of pure $\Q$-Hodge structures forms a $\Q$-linear Tannakian category with fiber functor the forgetful functor to $\Q$-vector spaces.
The Mumford-Tate group $G$ of a Hodge structure $H$ can also be defined as the Tannaka group of the smallest Tannakian subcategory $\langle H \rangle^{\otimes}$ of $\Q-\mathrm{HS}$ containing $H$.
\end{enumerate}
\end{rmk}
Composing $\rho: \mathbb{S} \to G_\R$ with
$$\mathbb{G}_{m, \C} \to \mathbb{S}_\C = \mathbb{G}_{m, \C} \times \mathbb{G}_{m, \C} , \,\,\, t \mapsto (t,1) $$
gives a cocharacter $\mu: \mathbb{G}_{m,\C} \to G_\C$, called the \emph{Hodge cocharacter} of $H$.
The associated filtration is the \emph{Hodge filtration} $F^{\bullet}$ on $H \otimes_\Q \C$ defined by $$F^{i}:= \bigoplus_{p \ge i} H^{p,q}. $$
For a pure Hodge structure of weight $n$, one can recover the Hodge decomposition from the Hodge filtration by the formula $$H^{p,q} = F^p \cap \overline{F^q}. $$

Denote by $P_{\mu} \subset G_{\C}$ the parabolic subgroup attached to $\mu$ and define the flag variety $$\mathcal{F}_\C := G_{\C} / P_{\mu}. $$
By a classical argument (cf. \cite[Lemma 12.1]{MilneShimura}), the $G(\C)$-conjugacy class of $\mu$ is defined over $\overline \Q$, and therefore
$\mathcal{F}_\C$ has a model $\mathcal{F}$ over $\overline \Q$.

\begin{defn}\label{defnflagvariety}
We call the variety $\mathcal{F}$ the \emph{flag variety attached to the Hodge structure} $H$.
\end{defn}

In this section we define two invariants of a $\Q$-Hodge structure, the transcendence degree and the horizontal codimension.
The first notion depends crucially on the $\Q$-structure:

\begin{defn}\label{trdeg}
\begin{enumerate}[(i)]
\item
Let $H$ be a $\Q$-Hodge structure. We say that the Hodge filtration $F^{\bullet} \subset H\otimes_\Q \C$ can be defined over a subfield $K \subset \C$ if there is a filtration $F_K^{\bullet} \subset H \otimes_\Q K$ such that $F_K^{\bullet} \otimes_K \C = F^{\bullet}$.
\item 
The \emph{transcendence degree} $\mathrm{trdeg \,} H$ of the Hodge structure $H$ is the smallest integer $n$ such that the Hodge filtration can be defined over a subfield $K \subset \C$ with $\mathrm{trdeg}_{\Q} K = n$.
\end{enumerate}
\end{defn}

There is an obvious upper bound for the transcendence degree of a Hodge structure $H$ in terms of the associated flag variety.
Namely, the Hodge filtration $F^{\bullet}$ can be viewed as a complex point $ F^{\bullet} \in \mathcal{F}(\C)$, and therefore $\mathrm{trdeg \,} H \le \dim \mathcal{F}$ as $\mathcal{F}$ is defined over $\overline \Q$. If $\overline{\{F^{\bullet}\}}^{\overline \Q-\mathrm{Zar}} \subseteq \mathcal{F}$ denotes the smallest closed $\overline \Q$-subvariety of $\mathcal{F}$ containing $F^{\bullet}$, then $$\mathrm{trdeg \,} H = \dim \overline{\{F^{\bullet}\}}^{\overline \Q-\mathrm{Zar}}.$$

\begin{defn}\label{maxtrdeg}
We say a $\Q$-Hodge structure $H$ is \emph{of maximal transcendence degree} if $\mathrm{trdeg \,} H = \dim \mathcal{F}$.
Since $\mathcal{F}$ is connected, this is equivalent to $$\overline{\{F^{\bullet}\}}^{\overline \Q-\mathrm{Zar}} = \mathcal{F}.$$
\end{defn}

The second invariant of $H$ is a more geometric invariant.
By composing $\rho$ with the adjoint representation of the Mumford-Tate group $G$, its Lie algebra
$\mathfrak{g} = \mathrm{Lie \,} G$ carries a $\Q$-Hodge structure of weight zero
$$\mathfrak{g}_\C = \oplus_{i} \mathfrak{g}^{i,-i}.$$
We denote by $F^{i}\mathfrak{g}$ the associated Hodge filtration.

The Lie algebra of the parabolic subgroup $P_\mu$ is identified with $F^{0}\mathfrak{g}$.
Therefore the tangent space of the flag variety $\mathcal{F}_\C$
can be identified with $\mathfrak{g}_\C / F^{0}\mathfrak{g}$, and
we have $$\dim \mathcal{F} = \dim_{\C} \mathfrak{g}_\C - \dim_\C F^{0}\mathfrak{g}.$$

\begin{defn}\label{defnhorizontalcodimension}
The \emph{horizontal codimension} of the Hodge structure $H$ is defined to be $$\mathrm{hcodim \,} H := \dim_{\C} \mathfrak{g}_\C - \dim_\C F^{-1}\mathfrak{g}.$$
\end{defn}
\begin{rmk}
\begin{enumerate}[(i)]
\item Identifying the tangent space of the flag variety $\mathcal{F}_\C$
with $\mathfrak{g}_\C / F^{0}\mathfrak{g}$,
the subspace $T_h\mathcal{F}_\C := F^{-1}\mathfrak{g}/F^{0}\mathfrak{g} \subset \mathfrak{g}_\C / F^{0}\mathfrak{g}$ is often called the \emph{horizontal tangent space}.
With this terminology, $\mathrm{hcodim \,} H$ is the codimension of the horizontal tangent space in the full tangent space.
\item This is different from the notion of \emph{horizontal Hodge codimension} in \cite[Definition 5.8]{BKU}. Since this notion has been mostly replaced by the \emph{Hodge codimension} as defined in \cite[Definition 5.1]{BKU}, we hope this terminology does not cause any confusion.
\end{enumerate}
\end{rmk}

Note that $\mathrm{hcodim \,} H = 0$ if and only if the weights which occur in the adjoint representation for the Hodge cocharacter $\mu: \mathbb{G}_m \to G_\C$ are contained in $\{1,0,-1\}$. 

\subsection{André motives}\label{sectionAndre}

Fix a subfield $K\subseteq \C$. André defined a semi-simple $\Q$-linear Tannakian category $\Mot_K^{\And}$ of pure motives over $K$, relying on his notion of motivated cycles (cf. \cite[Definition 1]{AndreIHES}).
Its objects consist of triples $(X, p, n)$, where $X$ is a smooth projective algebraic variety over $K$, $p \in \mathrm{Cor}_{\mot}^{0}(X,X)$ is an idempotent motivated correspondence of degree $0$ and $n \in \Z$ is an integer.
Morphisms in $\Mot_K^{\And}$ are given by motivated correspondences \cite[4.2]{AndreIHES}.
We use the notation $\mathfrak{h}(X) := (X, \Delta_X, 0)$, where $\Delta_X \subset X \times X$ denotes the diagonal.
We refer to \cite{AndreIHES} for the precise definition and detailed discussion of these concepts.

The \emph{Betti realization} defines a $\Q$-linear fiber functor
$$\omega_\B: \Mot_K^{\And} \to \mathrm{Vect}_\Q.$$
This functor sends the motive $\mathfrak{h}(X)$ of a smooth projective variety $X$ to $\HH_\B^{*}(X, \Q)$. It depends crucially on the fixed embedding $K \hookrightarrow \C$. 
Taking the Hodge structure on Betti cohomology into account, one can refine this functor to an exact faithful tensor functor 
$$ \omega_\Hdg: \Mot_K^{\And} \to \Q-\mathrm{HS}$$
to the category of pure $\Q$-Hodge structures, called the \emph{Hodge realization}.

\begin{defn}
We say that an André motive $M$ is \emph{pure} (of weight $n$) if $\omega_\Hdg(M)$ is a pure Hodge structure of weight $n$.
\end{defn}

Similarly, the \emph{de Rham realization} defines a fiber functor
$$ \omega_\dR: \Mot_K^{\And} \to \mathrm{Vect}_K.$$
The Hodge filtration on de Rham cohomology defines a filtration on the fiber functor $\omega_\dR$.

\begin{defn}
The \emph{motivic Galois group} of a motive $M \in \Mot_K^{\And}$
 is the algebraic group $$G^{\And}(M) := \mathrm{Aut}^{\otimes}( {\omega_\B}|_{\langle M\rangle^{\otimes}})$$ of tensor automorphisms of the fiber functor $\omega_\B$ restricted to $\langle M\rangle^{\otimes}$. Here $\langle M\rangle^{\otimes}$ denotes the smallest Tannakian subcategory of $\Mot_K^{\And}$ containing $M$.
 
\end{defn}
\begin{rmk}
Since the category $\Mot_K^{\And}$ is semi-simple, the motivic Galois group $G^{\And}(M)$ of $M$ is a reductive group over $\Q$. It is conjectured, but not known to be connected. 
\end{rmk}

The following conjecture is a weaker form of the Hodge conjecture:

\begin{conj}[Hodge classes are motivated]\label{conjmotivated}
The functor 
$$ \omega_\Hdg: \Mot_\C^{\And} \to \Q-\mathrm{HS} $$
is fully faithful.
\end{conj}
\begin{rmk}
\begin{enumerate}[(i)]
\item Since the category $\Q-\mathrm{HS}$ is semi-simple, Conjecture \ref{conjmotivated} also implies that the image of $\omega_\Hdg$ is stable under subobjects.
\item 
Let $M$ be an André motive over $\C$ and $G$ denote the Mumford-Tate group of $\omega_\Hdg(M)$. The tensor functor $\omega_\Hdg$ induces an inclusion of Tannaka groups $G \subseteq G^{\And}(M)$.
Conjecture \ref{conjmotivated} is equivalent to the assertion that this inclusion is an equality for all $M$.
\end{enumerate}
\end{rmk}
For a smooth projective variety $X$ over $K$, integration of differential forms over topological cycles on $X_\C^{\an}$ defines an isomorphism $$c_X:  \HH_\dR^*(X/K) \otimes_K \C \cong \HH^*_\B(X, \Q) \otimes_\Q \C.$$
More generally, for any André motive $M$ over $K$ there is a natural complex comparison isomorphism
$$c_M:   \omega_\dR(M) \otimes_K \C \cong \omega_\B(M) \otimes_\Q \C$$
which recovers the previous isomorphism for motives of the form $M= \mathfrak{h}(X)$. Moreover, $c_M$ is functorial and compatible with tensor products, and thus defines an isomorphism of tensor functors 
$$c: \omega_\dR \otimes_K \C  \cong \omega_\B \otimes_\Q \C$$ on $\Mot^{\And}_K$.
Under this complex comparison, the filtration on the fiber functor $\omega_\dR$ is identified with the Hodge filtration of the Hodge realization. 

\begin{defn}\label{defntorsor}
Let $M$ be an André motive over $K$. 
The \emph{motivated period torsor} of $M$ is defined to be the torsor of tensor isomorphisms 
$$\Omega^{\And}_M := \mathrm{Isom}^{\otimes}(\omega_\dR|_{\langle M\rangle^{\otimes}}, \omega_\B \otimes_\Q K |_{\langle M\rangle^{\otimes}}). $$
It is naturally a torsor under $G^{\And}(M)_K:= G^{\And}(M)\otimes_\Q K$.
\end{defn}
The complex comparison isomorphism $c_M:   \omega_\dR(M) \otimes_K \C \cong \omega_\B(M) \otimes_\Q \C$ defines a complex point $c_M \in \Omega^{\And}_M(\C)$.
A fundamental conjecture of Grothendieck asserts that the point $c_M$ ought to be a generic point of this torsor:

\begin{conj}[Grothendieck period conjecture (motivated version)]\label{conjgrothendieckperiod}
Let $M$ be an André motive over $\overline \Q$. Then
$$\overline{\{c_M\}}^{\overline \Q-\mathrm{Zar}} =  \Omega^{\And}_M.$$
\end{conj}
It is crucial for the Grothendieck period conjecture that we work with a motive over $\overline \Q$.
To allow more general base fields, André formulated the following generalization:

\begin{conj}[André's generalized Grothendieck period conjecture, (\cite{AndreBook}, 23.4.1)] \label{conjgeneralizedperiod}
Let $M$ be an André motive over a subfield $K \subset \C$ of finite transcendence degree over $\Q$. 
Denote by $\overline{\{c_M\}}^{K-\mathrm{Zar}} \subseteq \Omega^{\And}_M$ the $K$-Zariski closure of $c_M \in \Omega^{\And}_M(\C)$.
Then
$$\dim_K \overline{\{c_M\}}^{K-\mathrm{Zar}} \ge \dim_K \Omega^{\And}_M - \mathrm{trdeg}_\Q K$$
\end{conj}
The original formulation in (\cite{AndreBook}, 23.4.1) reads
$$\mathrm{trdeg}_K K(\textnormal{periods of } M) \ge \dim G^{\And}(M) - \mathrm{trdeg}_\Q K.$$
The left hand side can be identified with the dimension of  $\overline{\{c_M\}}^{K-\mathrm{Zar}}$, and since $\Omega^{\And}_M$ is a torsor under $G^{\And}(M)_K$, the right hand sides agree.

\begin{rmk}
In the special case $K = \overline \Q$, the generalized period conjecture predicts that $$\dim_{\overline \Q} \overline{\{c_M\}}^{\overline \Q-\mathrm{Zar}} = \dim_{\overline \Q} \Omega^{\And}_M.$$
That is, $\overline{\{c_M\}}^{\overline \Q-\mathrm{Zar}}$ is a connected component of $\Omega^{\And}_M$. 
Thus Conjecture \ref{conjgeneralizedperiod} for $K = \overline \Q$ does not quite imply the Grothendieck period conjecture
\ref{conjgrothendieckperiod}, as the latter also contains the statement that the torsor $\Omega^{\And}_M$ is connected. 
\end{rmk}

\section{Transcendence of the Hodge filtration}

\subsection{Hodge structures of varieties over number fields}

In this section we point out the consequences of the Grothendieck period conjecture \ref{conjgrothendieckperiod} for the Hodge structures of varieties over number fields.
Theorem \ref{Theorem} below is certainly known to the experts, in the special case where $M$ is the motive of an abelian variety it can be found in \cite[Proposition 23.2.4.1]{AndreBook}, and in the general case it is formulated as a conjecture in \cite[VIII.A.8]{MTdomains}.

Let $M$ be an André motive over $\overline \Q$. Its Betti realization $H:= \omega_\Hdg(M)$ is a $\Q$-Hodge structure. For the invariants attached to $H$ we use the notation introduced in section \ref{sectioninvariants}. In particular, we denote by $G$ the Mumford-Tate group of $H$, by $\mu: \mathbb{G}_{m, \C} \to G_\C$ the Hodge cocharacter, and define the associated flag variety $\mathcal{F}$ as Definition \ref{defnflagvariety}. 

In addition, we need a second flag variety $\mathcal{F}^{\And}$ which is a homogeneous space for the motivic Galois group $G^{\And}(M)$.
Since $G\subseteq G^{\And}(M)$, we can view the Hodge cocharacter $\mu$ as a cocharacter $\mu: \mathbb{G}_{m, \C} \to G^{\And}(M)_{\C}$.
Denote by $P_{\mu}^{\And} \subseteq G^{\And}(M)_{\C}$ the associated parabolic subgroup.
We define the flag variety $$\mathcal{F}^{\And}_{\C} := G^{\And}(M)_{\C} / P_\mu^{\And}.$$
By the same argument as before, it has a model $\mathcal{F}^{\And}$ over $\overline \Q$.
Moreover, we have a closed immersion $\mathcal{F} \subseteq \mathcal{F}^{\And}$ defined over $\overline \Q$.

Recall that we denote by $ \overline{\{F^{\bullet}\}}^{\overline \Q-\mathrm{Zar}} \subseteq \mathcal{F}$ the $\overline \Q$-Zariski closure of the point $F^{\bullet} \in \mathcal{F}(\C)$ defined by the Hodge filtration $F^{\bullet} \subseteq H \otimes_\Q \C$.

\begin{thm}[cf. {\cite[Proposition 23.2.4.1]{AndreBook}}] \label{Theorem} 
Assume the Grothendieck period conjecture \ref{conjgrothendieckperiod} holds for the André motive $M$ over $\overline \Q$. Then $$ \overline{\{F^{\bullet}\}}^{\overline \Q-\mathrm{Zar}} = \mathcal{F}^{\And}.$$
\end{thm}
\begin{proof}
The Hodge filtration $F^{\bullet}_{\dR} \subset \omega_{\dR}(M)$ on the de Rham realization of $M$ is defined over $\overline \Q$.
Denote by $\Omega_M^{\And}$ the motivated period torsor of $M$ (cf. Definition \ref{defntorsor}), a torsor under $G^{\And}(M)_{\overline \Q}$.
An element $\varphi \in \Omega_M^{\And}(\overline \Q)$ can be used to define a filtration $\varphi( F^{\bullet}_\dR)$ on $\omega_{\B}(M) \otimes_{\Q} \overline \Q$. Since $\varphi$ is an isomorphism of tensor functors, this defines a filtration on the fiber functor $\omega_\B \otimes _{\Q} \overline \Q$ restricted to $\langle M\rangle^{\otimes}$
and therefore a point of $\mathcal{F}^{\And}(\overline \Q)$. 
The morphism $$\Phi: \Omega_M^{\And} \to \mathcal{F}^{\And}, \,\, \varphi \mapsto \varphi( F^{\bullet}_\dR)$$
 is a surjective $G^{\And}(M)_{\overline \Q}$-equivariant morphism of varieties over $\overline \Q$.
The complex comparison isomorphism $$c_M: \omega_{\dR}(M) \otimes \C \cong \omega_\B(M)\otimes \C$$ maps $F^{\bullet}_{\dR}$ to the Hodge filtration $F^{\bullet}$ on $\omega_\B(M) \otimes_\Q \C$.
Consequently, if we view $c_M \in \Omega_M^{\And}(\C)$ as a complex point of the period torsor, then $\Phi(c_M) = F^{\bullet} \in \mathcal{F}^{\And}(\C)$.
The Grothendieck period conjecture predicts that $\overline{\{c_M\}}^{\overline \Q-\mathrm{Zar}} =  \Omega_M^{\And}$, and therefore $\overline{\{F^{\bullet}\}}^{\overline \Q-\mathrm{Zar}} = \mathcal{F}^{\And}$.
\end{proof}
\begin{rmk}
We want to point out a special case of Theorem \ref{Theorem} which is known unconditionally.
Assume that $M$ is an abelian motive, i.e. it lies in the Tannakian subcategory of $\Mot^{\And}_{\overline \Q}$ generated by motives of abelian varieties. 
Shiga-Wolfart and Cohen proved the following theorem:

\begin{thm}[{\cite[Main Theorem]{ShigaWolfart}, \cite{Cohen}}] \label{SWC}
Let $M$ be an abelian motive over $\overline \Q$.
If the Hodge filtration on $\omega_\B(M) \otimes_\Q \C$ is defined over $\overline \Q$, 
then the motive $M$ has \emph{complex multiplication}, i.e. the group $G^{\And}(M)$ is a torus.
\end{thm}

In the language of Theorem \ref{Theorem}, this assumption on the Hodge filtration translates into the fact that
$\overline{\{F^{\bullet}\}}^{\overline \Q-\mathrm{Zar}}$ is a point.
Theorem \ref{Theorem} implies that $ \mathcal{F}^{\And}$ is a point, which happens only if $G^{\And}(M)$ is a torus.
Hence Theorem \ref{SWC} corresponds to the special case of Theorem \ref{Theorem} where $\overline{\{F^{\bullet}\}}^{\overline \Q-\mathrm{Zar}}$ is zero-dimensional, and $M$ is an abelian motive.
\end{rmk}

\begin{cor}
Let $M$ be an André motive over $\overline \Q$. Assume the Grothendieck period conjecture \ref{conjgrothendieckperiod} holds for $M$. 
 Then the Hodge structure $\omega_\Hdg(M)$ is of maximal transcendence degree in the sense of Definition \ref{maxtrdeg}. 
\end{cor}
\begin{proof}
We have $F^{\bullet} \in \mathcal{F}(\C) \subseteq \mathcal{F}^{\And}(\C)$.
Theorem \ref{Theorem} implies that $\overline{\{F^{\bullet}\}}^{\overline \Q-\mathrm{Zar}} = \mathcal{F} = \mathcal{F}^{\And}$.
Hence $\omega_\Hdg(M)$ is of maximal transcendence degree.
\end{proof}

This gives a (conditional) purely Hodge theoretic criterion for showing that a complex algebraic variety is not defined over a number field:
Assume the Grothendieck period conjecture \ref{conjgrothendieckperiod}. Let $X$ be a smooth projective complex algebraic variety such that for some $n$, the Hodge structure $\HH_\B^n(X, \Q)$ is not of maximal transcendence degree.
Then $X$ does not have a model over $\overline \Q$.

\subsection{Hodge structures and André's generalized period conjecture}

We keep the notation of the previous section, but more generally allow $M$ to be an André motive over a finitely generated subfield $K \subset \C$.
Thus $H:= \omega_\Hdg(M)$ still denotes the Hodge realization of $M$, with associated flag variety $\mathcal{F}$, and the flag variety $\mathcal{F}^{\And}$ for $M$ is defined as in the previous section.

Using André's generalized Grothendieck period conjecture \ref{conjgeneralizedperiod}, we have the following generalization of Theorem \ref{Theorem}:

\begin{thm}[{cf. \cite[23.4.4]{AndreBook}}]\label{theoremgeneralized}
Let $K \subset \C$ be a finitely generated subfield and $M$ an André motive over $K$.
Denote by $H := \omega_\Hdg(M)$ the Hodge realization of $M$.
Assume André's generalized period conjecture holds for $M$.
Then $$\trdeg H \ge \dim \mathcal{F}^{\And} - \mathrm{trdeg}_\Q K.  $$
\end{thm}
\begin{proof}
The proof is similar to the one of Theorem \ref{Theorem}, replacing the Grothendieck period conjecture \ref{conjgrothendieckperiod} by André's generalized Grothendieck period conjecture \ref{conjgeneralizedperiod}.
The motivated period torsor $\Omega_M^{\And}$ is now defined over $K$ and a torsor under $G^{\And}(M)_{K}$.
We still have a surjective morphism defined over $K$
$$\Phi: \Omega_M^{\And} \to \mathcal{F}^{\And} \otimes K, \,\, \varphi \mapsto \varphi( F^{\bullet}_\dR).$$ 
The complex comparison  $c_M \in \Omega_M^{\And}(\C)$ maps to $\Phi(c_M) = F^{\bullet} \in \mathcal{F}^{\And}(\C)$. Denoting by $\overline{\{F^{\bullet}\}}^{K-\mathrm{Zar}} \subseteq \mathcal{F}^{\And} \otimes K$ the $K$-Zariski closure of $F^{\bullet}$, 
it follows that $\Phi(\overline{\{c_M\}}^{K-\mathrm{Zar}})$ is dense in
$\overline{\{F^{\bullet}\}}^{K-\mathrm{Zar}} $.
Since all fibers of $\Phi$ have dimension $\dim P^{\And}_\mu$, it follows that
$$\dim_K \overline{\{F^{\bullet}\}}^{K-\mathrm{Zar}} \ge \dim_K \overline{\{c_M\}}^{K-\mathrm{Zar}} - \dim P^{\And}_\mu.$$
By André's generalized Grothendieck period conjecture \ref{conjgeneralizedperiod}, 
%(just $K$ and not $\overline K$?, dimension of Omega = dimension of $G^{\And}$?)
$$ \dim_K \overline{\{c_M\}}^{K-\mathrm{Zar}} \ge \dim G^{\And}(M) - \mathrm{trdeg}_\Q K.$$
We conclude that \begin{eqnarray*} \trdeg H  =   \dim_\Q \overline{\{F^{\bullet}\}}^{\overline \Q-\mathrm{Zar}} & \ge & \dim_K \overline{\{F^{\bullet}\}}^{K-\mathrm{Zar}} \\
 & \ge &  \dim_K \overline{\{c_M\}}^{K-\mathrm{Zar}} - \dim P^{\And}_\mu \\
  & \ge & \dim G^{\And}(M) - \mathrm{trdeg}_\Q K - \dim P^{\And}_\mu \\
  & = & \dim \mathcal{F}^{\And} - \mathrm{trdeg}_\Q K.
 \end{eqnarray*}

\end{proof}
\begin{rmk}
In the case where $M$ is the motive of an abelian variety, Theorem \ref{theoremgeneralized} can be found in \cite[23.4.4]{AndreBook}.
\end{rmk}
As $\mathcal{F} \subseteq \mathcal{F}^{\And}$, one obtains:
\begin{cor}\label{corgeneralized}
Let $K \subset \C$ be a finitely generated subfield and $M$ an André motive over $K$ with Hodge realization $H$.
Assume André's generalized period conjecture \ref{conjgeneralizedperiod} holds for $M$.
Then $$\trdeg H \ge \dim \mathcal{F} - \mathrm{trdeg}_\Q K.  $$
\end{cor}
 
\section{Hodge structures coming from geometry}

\subsection{Hodge structures not coming from geometry}

Let $M$ be a pure André motive over $\C$, and $H := \omega_\Hdg(M)$ the associated $\Q$-Hodge structure. In this section, we prove an upper bound for the transcendence degree of a field of definition of $M$ in terms of data attached to the Hodge structure $H$.
Denote by $G$ the Mumford-Tate group of $H$.
Recall (from section \ref{sectioninvariants}) that the Hodge cocharacter defines a Hodge filtration $F^{\bullet} \mathfrak{g}$ on the (complex) Lie algebra $\mathfrak{g}_\C$ of $G$.

\begin{prop}\label{propdescent}
Assume that Hodge cycles are motivated (Conjecture \ref{conjmotivated}). 
Let $M$ be a pure André motive over $\C$. Denote by $H := \omega_\Hdg(M)$ the associated $\Q$-Hodge structure with Mumford-Tate group $G$.
Then $M$ is isomorphic to the base change of an André motive defined over a subfield $K \subset \C$ with $$\mathrm{trdeg}_\Q K \le \dim_\C F^{-1}\mathfrak{g} -\dim_\C F^{0} \mathfrak{g}. $$
\end{prop}
\begin{proof}
Following \cite[Remarques (ii) after Théorème 5.2]{AndreIHES}, we can spread out the motive $M$ as follows:
$M$ can be defined over a finitely generated extension $L$ of $\Q$.
There is a smooth quasi-projective variety $S$ defined over a number field such that $S$ has function field $L$, and a family of motives $\mathbb{M} \to S$ defined over $\overline \Q$ in the sense of \cite[§5.2]{AndreIHES} such that the fiber $ \mathbb{M}_s=M$ for some $s \in S(\C)$.
Furthermore, the generic motivic Galois group of $\mathbb{M}$ is $G^{\And}(M)$.

Let $\mathbb{V}$ denote the variation of Hodge structure over $S_\C$ attached to $\mathbb{M}$. If we assume Conjecture \ref{conjmotivated}, then $G^{\And}(M) = G$, and hence the generic Mumford-Tate group of $\mathbb{V}$ is also $G$. Denote by 
$$\varphi: S_\C^{\an} \to \mathcal{D} / \Gamma  $$ the associated period map, where the target is the quotient of a Mumford-Tate domain for the group $G$ by a discrete group. Hence $\mathcal{D} \subseteq \mathcal{F}_\C^\an$ is an open complex analytic subvariety of the flag variety attached to $H$.
By \cite[Theorem 1]{BBT}, there is a factorization
$$ \varphi: S_\C^{\an} \to B_\C^{\an} \to Y^{\an} \subset \mathcal{D} / \Gamma, $$
where 
\begin{enumerate}
    \item $f: S \to B$ is a morphism of algebraic varieties over $\overline \Q$ with connected fibers,
    \item
    $B_\C \to Y$ is a quasi-finite morphism of algebraic varieties over $\C$,
    \item
    $Y^\an \subset \mathcal{D} / \Gamma$ is a closed immersion of analytic spaces,
\end{enumerate}
and such that the composition $S_\C \to B_\C \to Y$ is dominant.

We claim that $$\dim_{\overline \Q} B \le \dim_\C F^{-1}\mathfrak{g} - \dim_\C F^{0} \mathfrak{g}. $$
Namely, the tangent space of $\mathcal{D} / \Gamma$ at $\varphi(s)$ can be identified with $\mathfrak{g} / F^{0} \mathfrak{g}$, and the right hand side is precisely the dimension of the horizontal tangent space $T_{h, \varphi(s)}(\mathcal{D} / \Gamma) = F^{-1}\mathfrak{g} / F^{0} \mathfrak{g}$.
By Griffiths transversality, % $\mathrm{d}\varphi(T_s S) \subset T_{h, \varphi(s)}(\mathcal{D} / \Gamma)$.
$T_{\varphi(s)} Y\subset T_{h, \varphi(s)}(\mathcal{D} / \Gamma) $.
%As $Y \subset \mathcal{D} / \Gamma$ is a closed immersion, 
We get $\dim_\C Y \le \dim_\C F^{-1}\mathfrak{g} - \dim_\C F^{0} \mathfrak{g}$.
Since $B_\C$ is quasi-finite over $Y$, it follows that $\dim_{\overline \Q} B  \le \dim_\C F^{-1}\mathfrak{g} - \dim_\C F^{0} \mathfrak{g}$.

Let $b:= f(s) \in B(\mathbb{C})$. By the above dimension argument, $b$ is defined over a subfield $K \subset \C$ with $\mathrm{trdeg}_\Q K \le \dim_\C F^{-1}\mathfrak{g} - \dim_\C F^{0} \mathfrak{g}$.
We may assume that $K$ is algebraically closed.
Let $Z := f^{-1}(b) \subseteq S$, a connected closed algebraic subvariety defined over $K$.
We claim that for every $x \in Z(\C)$, the André motive $\mathbb{M}_x$ is isomorphic to $M$. This follows from André's Theorem of the fixed part \cite[Théorème 0.5]{AndreIHES}.
Namely, for every $x \in Z(\C)$, there is a smooth connected curve $C$ and a morphism $g: C \to Z$ whose image contains both $x$ and $s$.
We can pull back the family of motives $\mathbb{M}$ to $C$, and note that the associated variation of Hodge structure $\mathbb{V}_{|C}$ is constant (and so, up to taking a finite cover, the underlying local system).
The argument in the proof of \cite[Théorème 0.5]{AndreIHES} shows that $\mathbb{M}_x \cong \mathbb{M}_s =M$. Indeed, both are isomorphic to the fixed part of the family of motives $\mathbb{M}_{|C}$, which is shown to be an André motive and independent of the point in $C(\C)$ in the proof of loc.cit. 
For any point $x \in Z(K)$, we conclude that $\mathbb{M}_x$ is an André motive over $K$ whose base change to $\C$ is isomorphic to $M$, as desired.
\end{proof}
\begin{rmk}
In the above proof, the assumption that Hodge cycles are motivated is needed to ensure that we can spread out $M$ to a family of motives $\mathbb{M}$ defined over $\overline \Q$ which has generic Mumford-Tate group equal to $G$.
We do not need the full power of this assumption for this, it is enough to know that special subvarieties for families of motives over $\overline \Q$ are defined over $\overline \Q$.
For example, if the Mumford-Tate group $G$ is a torus, then Proposition \ref{propdescent} predicts that $M$ can be defined over $\overline \Q$, which would follow from the definability of CM points over $\overline \Q$.
\end{rmk}

\begin{thm}\label{thmHSnotgeometric}
Let $M$ be a pure André motive over $\C$ and $H := \omega_\Hdg(M)$ the associated $\Q$-Hodge structure.
Assume that Hodge cycles are motivated (Conjecture \ref{conjmotivated}) and that André's generalized period conjecture \ref{conjgeneralizedperiod} holds true.
Then $$\trdeg H \ge \mathrm{hcodim \,} H.$$
 \end{thm}
\begin{proof}
By Proposition \ref{propdescent}, the André motive $M$ is isomorphic to the base change of an André motive $N$ defined over a field $K \subset \C$ with $$\mathrm{trdeg}_\Q K \le \dim_\C F^{-1}\mathfrak{g} -\dim_\C F^{0} \mathfrak{g}.$$
Applying Corollary \ref{corgeneralized} to $N$, we obtain
\begin{eqnarray*} \trdeg H & \ge & \dim \mathcal{F} - \mathrm{trdeg}_\Q K \\ 
& = & \dim_\C \mathfrak{g}_\C -\dim_\C F^{0} \mathfrak{g} - \mathrm{trdeg}_\Q K \\
& \ge & \dim_\C \mathfrak{g}_\C -\dim_\C F^{0} \mathfrak{g} - (\dim_\C F^{-1}\mathfrak{g} -\dim_\C F^{0} \mathfrak{g}) \\
& = & \mathrm{hcodim \,} H. \end{eqnarray*}
\end{proof}

\begin{rmk}
    Let us emphasize that the origin of the non-trivial bound in Theorem \ref{thmHSnotgeometric} is Griffiths transversality.
    Indeed, without taking Griffiths transversality into account, Proposition \ref{propdescent} would produce the more naive bound $\mathrm{trdeg}_\Q K \le \dim_\C \mathfrak{g}_\C -\dim_\C F^{0} \mathfrak{g}$. Plugging this into the proof of Theorem \ref{thmHSnotgeometric} yields nothing but the trivial bound $\trdeg H \ge 0$.
\end{rmk}

\subsection{Distribution of admissible Hodge structures}\label{distribution}

Let $H$ be a pure $\Q$-Hodge structure with Mumford-Tate group $G$ and $\rho: \mathbb{S} \to \mathrm{GL}(H_\R)$ the associated representation of the Deligne torus.
The \emph{Mumford-Tate domain} of $H$ is the $G(\R)$-orbit of $\rho$, seen as an open complex analytic subvariety $ \mathcal{D} \subseteq \mathcal{F}_\C^{\an}$.

A point $x \in \mathcal{D}$ defines a $\Q$-Hodge structure $H_x$.
The classical argument mentioned in the introduction says the following:
As soon as $\mathrm{hcodim \,} H > 0$, the set of points $x \in \mathcal{D}$ for which the Hodge structure $H_x$ comes from geometry is contained in a countable union of strict closed analytic subvarieties.
Indeed, by spreading out, every André motive over $\C$ lies in one of the countably many families of motives defined over $\overline \Q$.
For every such family, the locus where the image of the period map meets $\mathcal{D}$ is contained in a strict closed analytic subvariety. 
This can be seen from the fact that Griffiths transversality is a non-trivial constraint for $\mathrm{hcodim \,} H > 0$.

In contrast to that, we want to study the set of points in $x \in \mathcal{D}$ which satisfy the condition \begin{equation}\label{InequalityDistribution} \trdeg H_x < \mathrm{hcodim \,} H_x.\end{equation} Since points lying in the same $G(\Q)$-orbit define isomorphic $\Q$-Hodge structures, this set is stable under the action of $G(\Q)$.

The Mumford-Tate group of $H_x$ is a subgroup $G_x \subseteq G$.
We denote by $\mathcal{D}^\circ \subseteq \mathcal{D}$ the set of \emph{Hodge generic points}, i.e. points where this inclusion is an equality.
The complement of the set of Hodge generic points is the intersection of $\mathcal{D}$ with countably many closed algebraic subvarieties of $\mathcal{F}$, which are flag varieties attached to smaller Mumford-Tate groups.

For all $x \in \mathcal{D}^\circ$, we have $\mathrm{hcodim \,} H_x = \mathrm{hcodim \,} H$.
Therefore the set of points $x \in \mathcal{D}^{\circ}$ which satisfy condition (\ref{InequalityDistribution}) is equal to
\begin{equation}\label{nonadmissibleset} \mathcal{D}^\circ \cap \bigcup_{\substack{Z \subseteq \mathcal{F}, \\ \dim Z < \mathrm{hcodim \,} H}} Z(\C). \end{equation}
Here the union runs over all closed $\overline \Q$-subvarieties of $\mathcal{F}$ of dimension less than $\mathrm{hcodim \,} H$.

This shows that the necessary criterion we formulate is far from sufficient:
By the above argument, we expect that the points of $\mathcal{D}$ corresponding to Hodge structures coming from geometry are contained in a countable union of strict closed analytic subvarieties.
However, our criterion only \emph{excludes} a countable union of strict closed subvarieties!

\subsection{Examples}\label{sectionexamples}

\subsubsection{CM Hodge structures}
A Hodge structure $H$ is called a \emph{CM Hodge structure} if the associated Mumford-Tate group $G$ is a torus .
For such Hodge structures, $\dim \mathcal{F}=0$, and consequently $\mathrm{hcodim \,} H = \trdeg H = 0$.
In particular, Theorem \ref{thmHSnotgeometric} causes no restriction in this case.
In fact, it is known that all (polarizable) CM Hodge structures come from geometry (cf. \cite[§7]{SerreProperties}, the proof is also spelled out in \cite[Proposition 4.6]{MilneFinite}).

\subsubsection{Hodge structures of Shimura type}

We have $\mathrm{hcodim \,} H=0$ if and only if $F^{-1} \mathfrak{g} = \mathfrak{g}_\C$. 
In this case we say that the Hodge structure is \emph{of Shimura type}.
In the terminology of \cite[Definition 4.13]{BKU}, this means that $H$ has "level one" (at least in the irreducible case).
For Hodge structures of Shimura type, Theorem \ref{thmHSnotgeometric} poses no restrictions on the Hodge structures coming from geometry.
This corresponds to the fact that Griffiths transversality is a trivial constraint. For example, all polarizable Hodge structures of type $(1,0), (0,1)$ come from abelian varieties.

Conversely, Theorem \ref{thmHSnotgeometric} implies:

\begin{cor} \label{CorShimuraType}
Assume that Hodge cycles are motivated and that André's generalized period conjecture holds true.
Let $H$ be a pure Hodge structure coming from geometry.
If the Hodge filtration $F^{\bullet} \subseteq H \otimes_\Q \C$ is defined over $\overline \Q$, then $H$ is of Shimura type.
\end{cor}

\subsubsection{Hodge structures of type $(n,0,...,0,n)$}

These Hodge structures were studied by Totaro in \cite{Totaro}.
Griffiths transversality does not allow positive dimensional families of Hodge structures of this type, and hence $\mathrm{hcodim \,} H = \dim \mathcal{F}$.
As was pointed out in the introduction of \cite{Totaro}, this means that only countably many such Hodge structures should come from geometry.
In this case, $H$ is of Shimura type if and only if it is a CM Hodge structure \cite[Corollary 4.2]{Totaro}. By the analysis in \cite[Theorem 3.1]{Totaro}, there exist examples which are not CM (e.g. of type $(2,0,2)$ with endomorphisms by a quaternion algebra over $\Q$).

It follows from Corollary \ref{CorShimuraType} that assuming the relevant conjectures, a Hodge structure of type $(n,0,...,0,n)$ whose Hodge filtration is defined over $\overline \Q$ does not come from geometry unless it is of CM type.  

Once again we see that our criterion is far from being sufficient, since it still allows uncountably many Hodge structures of type $(n,0,...,0,n)$.

\subsubsection{Calabi-Yau type Hodge structures}

In this example we consider rank four Hodge structures of weight $3$ of type $(1,1,1,1)$, i.e. with $\dim H^{0,3} = \dim H^{1,2}=1$.
A generic such Hodge structure $H$ has Mumford-Tate group $\mathrm{GSp}_4$, and the associated flag variety has dimension $\dim \mathcal{F} = 4$. A detailed description of the Mumford-Tate domain $\mathcal{D}$ and its Mumford-Tate subdomains can be found in \cite[Theorem VII.F.1]{MTdomains}.
One computes that $\mathrm{hcodim \,} H =2$. 
In this case, Theorem \ref{thmHSnotgeometric} predicts that any such generic $H$ which comes from geometry satisfies $\trdeg H \ge 2$.

\bibliographystyle{amsalpha}
\bibliography{bib}

\end{document}